\documentclass[runningheads]{llncs}

\usepackage{microtype,comment}
\usepackage{enumerate}
\usepackage{boxedminipage}
\usepackage{tikz,tkz-graph}
\usepackage{hyperref}
\usepackage{amssymb,amsmath}
\usepackage{slashbox}
\usepackage{xypic}
\usepackage[T1]{fontenc}

\newcommand{\dist}{{\mbox{dist}}}
\newcommand{\NP}{{\sf NP}}

\newcounter{ctrclaim}[theorem]
\newcounter{ctrcase}[theorem]

\title{Colouring Graphs of Bounded Diameter in the Absence of Small Cycles\thanks{Research supported by the Leverhulme Trust (RPG-2016-258). An extended abstract~\cite{MPS21} of the paper will appear in the proceedings of CIAC 2021.}}

\author{Barnaby Martin\inst{1} \and
Dani\"el Paulusma\inst{2}\orcidID{0000-0001-5945-9287}
\and Siani Smith\inst{1}}

\institute{Department of Computer Science, Durham University, UK
\email{\{barnaby.d.martin,daniel.paulusma,siani.smith\}@durham.ac.uk}}

\authorrunning{B. Martin, D. Paulusma and S. Smith}
\titlerunning{Colouring Graphs of Bounded Diameter in the Absence of Small Cycles}

\begin{document}
\maketitle

\begin{abstract}
For  $k\geq 1$, a $k$-colouring $c$ of $G$ is a mapping from $V(G)$ to $\{1,2,\ldots,k\}$ such that $c(u)\neq c(v)$ for any two non-adjacent vertices $u$ 
and~$v$. The {\sc $k$-Colouring} problem is to decide if a graph $G$ has a $k$-colouring. For a family of graphs ${\cal H}$, a graph~$G$ is ${\cal H}$-free if $G$ does not contain any graph from ${\cal H}$ as an induced subgraph. Let $C_s$ be the $s$-vertex cycle. In previous work (MFCS 2019) we examined the effect of bounding the diameter on the complexity of {\sc $3$-Colouring} for $(C_3,\ldots,C_s)$-free graphs and $H$-free graphs where $H$ is some polyad. Here, we prove for certain small values of $s$ that {\sc $3$-Colouring} is polynomial-time solvable for $C_s$-free graphs of diameter~$2$  and $(C_4,C_s)$-free graphs of diameter~$2$. In fact, our results hold for the more general problem {\sc List $3$-Colouring}.
We complement these results with some hardness result for diameter~$4$.
\end{abstract}

\section{Introduction}

Graph colouring is a well-studied topic in Computer Science due to its wide range of applications. A {\it $k$-colouring} of a graph $G$ is a mapping $c:V(G)\to \{1,\ldots,k\}$ that assigns each vertex $u$ a {\it colour} $c(u)$ in such a way that $c(u)\neq c(v)$ for any two adjacent vertices $u$ and $v$ of $G$. The aim is to find the smallest value of $k$ (also called the {\it chromatic number}) such that $G$ has a $k$-colouring. 
The corresponding decision problem is called {\sc Colouring}, or {\sc $k$-Colouring} if $k$ is fixed, that is, not part of the input. 
As even {\sc $3$-Colouring} is \NP-complete~\cite{Lo73}, {\sc $k$-Colouring} and {\sc Colouring}
 have been studied for many special graph classes, as surveyed in, for example, \cite{Al93,C14,GJPS,JT95,KTV99,Pa15,RS04b,Tu97}. This holds in particular for  {\it hereditary} classes of graphs, which are the classes of graphs closed under vertex deletion. 
 
 It is well known and not difficult to see that a class of graphs is hereditary if and only if it can be characterized by a unique set ${\cal F}_{\cal G}$ of minimal forbidden induced subgraphs. In particular, a graph $G$ is {\it $H$-free} for some graph~$H$ if $G$ does not contain $H$ as an {\it induced} subgraph. The latter means that we cannot modify $G$ into $H$ by a sequence of vertex deletions. For a set of graphs $\{H_1,\ldots,H_p\}$, a graph $G$ is {\it $(H_1,\ldots,H_p)$-free} if $G$ is $H_i$-free for every $i\in \{1,\ldots,p\}$. 

We continue a long-term study on the complexity of {\sc $3$-Colouring} for special graph classes.
Let $C_t$ and $P_t$ be the cycle and path, respectively, on $t$ vertices.
The complexity of {\sc $3$-Colouring} for $H$-free graphs has not yet been classified; in particular this is still
is open for $P_t$-free graphs for every $t\geq 8$, whereas the case $t=7$ is polynomial~\cite{BCMSSZ18}. For $t\geq 3$, let $C_{>t}=\{C_{t+1},C_{t+2},\ldots\}$. Note that for $t\geq 2$, the class of $P_t$-free graphs is a subclass of $C_{>t}$-free graphs.
Recently, Pilipczuk, Pilipczuk and Rz{\k{a}}\.{z}ewski~\cite{PPR} gave for every $t\geq 3$,
a quasi-polynomial-time algorithm for {\sc $3$-Colouring} on ${C}_{> t}$-free graphs. 
Rojas and Stein~\cite{RS} proved in another recent paper that for every odd integer $t\geq 9$, {\sc $3$-Colouring} is polynomial-time solvable for 
$({\cal C}^{odd}_{<t-3},P_t)$-free graphs, where ${\cal C}^{odd}_{<t}$ is the set of all odd cycles on less than $t$ vertices.
This complements a result from~\cite{GPS14}, which implies that for every $t\geq 1$, {\sc $3$-Colouring}, or more general {\sc List $3$-Colouring} (defined later), is polynomial-time solvable for $(C_4,P_t)$-free graphs (see also~\cite{LM17}).
 
The graph classes in this paper are only partially characterized by forbidden induced subgraphs: we also restrict the diameter.
The {\it distance} $\dist(u,v)$ between two vertices~$u$ and~$v$ in a graph~$G$ is the length (number of edges) of a shortest path between them. The {\it diameter} of a graph~$G$ is the maximum distance over all pairs of vertices in $G$.
Note that the $n$-vertex path~$P_n$ has diameter~$n-1$, but by removing an internal vertex the diameter becomes infinite.
Hence, for every integer~$d\geq 2$, the class of graphs of diameter at most~$d$ is not hereditary (whereas if $d=1$ we obtain the class of complete graphs, which is hereditary).

For every $d\geq 3$, the $3$-{\sc Colouring} problem for graphs of diameter at most~$d$ is \NP-complete, as shown by  Mertzios and Spirakis~\cite{MS16} who gave a highly non-trivial \NP-hardness construction for the case where $d=3$. In fact they proved that $3$-{\sc Colouring} is \NP-complete even for $C_3$-free graphs of diameter~$3$ and radius~$2$. The complexity of {\sc $3$-Colouring} for the class of all graphs of diameter~$2$ has been posed as an open problem in several papers~\cite{BKM12,BFGP13,MPS19,MS16,Pa15}. 

On the positive side, Mertzios and Spirakis~\cite{MS16} gave a subexponential-time algorithm for {\sc $3$-Colouring} on graphs of diameter~$2$. Moreover, as we discuss below, $3$-{\sc Colouring} is polynomial-time solvable for several subclasses of diameter~$2$.
In order to explain this, we need some terminology.

A graph $G$ has an {\it articulation neighbourhood} if $G-(N(v)\cup \{v\})$ is disconnected for some $v\in V(G)$. The neighbourhoods $N(u)$ and $N(v)$ of two distinct (and non-adjacent) vertices $u$ and $v$ are {\it nested} if $N(u)\subseteq N(v)$. The graph $K_{1,r}$ denotes the $(r+1)$-vertex {\it star}, that is, the graph with vertices $x,y_1,\ldots,y_r$ and edges $xy_i$ for $i=1,\ldots,r$.
The {\it subdivision} of an edge $uw$ in a graph removes $uw$ and replaces it with a new vertex $v$ and edges $uv$, $vw$. 
We let $K_{1,r}^\ell$ be the {\it $\ell$-subdivided star}, which is 
obtained from $K_{1,r}$ by subdividing {\it one} edge exactly $\ell$ times.
A {\it polyad} is a tree where exactly one vertex has degree at least~$3$.
The graph~$S_{h,i,j}$, for $1\leq h\leq i\leq j$, is the tree with one vertex~$x$ of degree~$3$ and exactly three leaves, which are of distance~$h$,~$i$ and~$j$ from~$x$, respectively.
Note that $S_{1,1,1}=K_{1,3}$. 
The {\it diamond} is obtained from the $4$-vertex complete graph by deleting an edge.

The {\sc $3$-Colouring} problem is polynomial-time solvable for:

\begin{itemize}
\item[$\bullet$] diamond-free graphs of diameter~$2$ with an articulation neighbourhood but without nested neighbourhoods~\cite{MS16};\\[-9pt]
\item[$\bullet$]  $(C_3,C_4)$-free graphs of diameter~$2$~\cite{MPS19};\\[-9pt]
\item[$\bullet$]  $K^2_{1,r}$-free graphs of diameter~$2$, for every $r\geq 1$~\cite{MPS19}; and\\[-9pt]
\item[$\bullet$]  $S_{1,2,2}$-free graphs of diameter~$2$~\cite{MPS19}.
\end{itemize}
\noindent
It follows from results in~\cite{EHK98,Ho81,LK07} that without the diameter-$2$ condition, {\sc $3$-Colouring} is \NP-complete again in each of the above cases; in particular {\sc $3$-Colouring} is \NP-complete for ${\cal C}$-free graphs for any finite set ${\cal C}$ of cycles.

\subsection*{Our Results}

We aim to increase our understanding of the complexity of {\sc $3$-Colouring} for graphs of diameter~$2$. 
In~\cite{MPS19} we mainly considered {\sc $3$-Colouring} for graphs of diameter~$2$ with some forbidden induced subdivided star. In this paper, we continue this study by focussing on {\sc $3$-Colouring} for $C_s$-free or $(C_s,C_t)$-free graphs of diameter~$2$ for small values of $s$ and $t$; in particular for the case where $s=4$ (cf. the aforementioned polynomial-time result for $(C_4,P_t)$-free graphs).
In fact we prove our results for a more general problem, namely {\sc List $3$-Colouring}, whose complexity for diameter~$2$ is also still open.
 A {\it list assignment} of a graph $G=(V,E)$ is a function $L$ that prescribes a {\it list of admissible colours} $L(u)\subseteq \{1,2,\ldots\}$ to each $u\in V$.
A colouring $c$  {\it respects}~${L}$ if  $c(u)\in L(u)$ for every $u\in V.$ 
For an integer~$k\geq 1$, if $L(u)\subseteq \{1,\ldots,k\}$ for each $u\in V$, then $L$ is a {\it list $k$-assignment}.
The {\sc List $k$-Colouring} problem is to decide if a graph $G$ with an list $k$-assignment $L$ has a colouring that respects $L$. 
If every list is $\{1,\ldots,k\}$, we obtain {\sc $k$-Colouring}.

The following two theorems summarize our main results.

\begin{theorem}\label{t-main}
For $s\in \{5,6\}$, {\sc List $3$-Colouring} is polynomial-time solvable for $C_s$-free graphs of diameter~$2$.
\end{theorem}

\begin{theorem}\label{t-main2}
For $t\in \{3,5,6,7,8,9\}$, {\sc List $3$-Colouring} is polynomial-time solvable for $(C_4,C_t)$-free graphs of diameter~$2$.
\end{theorem}

\noindent
The case $t=3$ in Theorem~\ref{t-main2} directly follows from the Hoffman-Singleton Theorem~\cite{HS60}, which states that there are only four $(C_3,C_4)$-free graphs of diameter~$2$. The cases $t\in \{5,6\}$ immediately follows from Theorem~\ref{t-main}. Hence, apart from proving Theorem~\ref{t-main}, we only need to prove Theorem~\ref{t-main2} for $t\in \{7,8,9\}$. 

We prove Theorem~\ref{t-main} and the case $t=7$ of Theorem~\ref{t-main2} in Section~\ref{s-pro}. As we explain in the same section, all these results follow from the same technique, which is based on a number of (known) propagation rules. We first colour a small number of vertices and then start to apply the propagation rules exhaustively. This will reduce the sizes of the lists of the vertices. The novelty of our approach is the following: we can prove that the diameter-2 property ensures such a widespread reduction that each precolouring changes our instance into an instance of {\sc 2-List Colouring}: the polynomial-solvable variant of {\sc List Colouring} where each list has size at most~$2$~\cite{Ed86} (see also Section~\ref{s-pre}). 

We prove the cases $t=8$ and $t=9$ of Theorem~\ref{t-main2} in Section~\ref{s-pro2} using a refinement of the technique from Section~\ref{s-pro}. We explain this refinement in detail at the start of Section~\ref{s-pro2}. In short, in our branching, we exploit information from earlier obtained no-answers to reduced instances of our original instance $(G,L)$.

We complement Theorems~\ref{t-main} and~\ref{t-main2} by  the following result for diameter~$4$, whose proof can be found in Section~\ref{a-hard}. 

\begin{theorem}\label{t-hard}
For every even integer $t\geq 6$, {\sc $3$-Colouring} is \NP-complete on the class of $(C_4,C_6,\ldots,C_t)$-free graphs of diameter~$4$.
\end{theorem}

\noindent
Results of Damerell~\cite{D73} imply that {\sc $3$-Colouring} is polynomial-time solvable for $(C_3,C_4,C_5,C_6)$-free graphs of diameter~$3$ and for $(C_3,\ldots, C_8)$-free graphs of diameter~$4$~\cite{MPS19}.
We were not able to reduce the diameter in Theorem~\ref{t-hard} from $4$ to $3$; see Section~\ref{s-con} for a further discussion, including other open problems.

\section{Preliminaries}\label{s-pre}

In this section we give some more terminology and notation. We also recall some useful result from the literature.

Let $G=(V,E)$ be a graph. A vertex $u\in V$ is {\it dominating} if $u$ is adjacent to every other vertex of $G$. 
For $S\subseteq V$, the graph $G[S]=(S,\{uv\; |\; u,v\in S\; \mbox{and}\; uv\in E\})$ denotes the subgraph of $G$ induced by $S$.  
The {\it neighbourhood} of a vertex $u\in V$ is the set $N(u)=\{v\; |\; uv\in E\}$ and the {\it degree} of $u$ is the size of $N(u)$.
For a set $U\subseteq V$, we write $N(U)=\bigcup_{u\in U}N(u)\setminus U$.

A {\it clique} is a set of pairwise adjacent vertices, and an {\it independent set} is a set of pairwise non-adjacent vertices. A graph is {\it complete} if its vertex set is a clique. We denote the complete graph on $r$ vertices by $K_r$. 
Recall that the diamond is the graph obtained from the $K_4$ after removing an edge. 
The {\it bull} is the graph obtained from a triangle on vertices $x,y,z$ after adding two new vertices $u$ and $v$ and edges $xu$ and $yv$.

Let $G$ be a graph with a list assignment $L$.
If $|L(u)|\leq \ell$ for each $u\in V$, then $L$ is a {\it $\ell$-list assignment}. 
A list $k$-assignment is a $k$-list assignment, but the reverse is not necessarily true.
The {\sc $\ell$-List Colouring} problem is to decide if a graph $G$ with an $\ell$-list assignment $L$ has a colouring that respects $L$. 
We use a known general strategy for obtaining a polynomial-time algorithm for {\sc List $3$-Colouring} on some class~${\cal G}$. That is, we will reduce the input to a polynomial number of instances of $2$-{\sc List Colouring} and use a well-known result due to Edwards.

\begin{theorem}[\cite{Ed86}]\label{t-2sat}
The {\sc $2$-List Colouring} problem is linear-time solvable.
\end{theorem}

\medskip
\noindent
We also need an observation.

\begin{lemma}\label{l-cycle}
Let $G$ be a non-bipartite graph of diameter~$2$. Then $G$ contains a $C_3$ or induced $C_5$.
\end{lemma}

\begin{proof}
As $G$ is non-bipartite, $G$ has an odd cycle. Let $C$ be an odd cycle in $G$ of minimum length. Then $C$ is induced; otherwise we would find a shorter odd cycle.
For contradiction, suppose that $C$ has length at least $7$. Consider two vertices $u$ and $v$ at distance~$3$ in $C$. Then $C$ contains a $4$-vertex path $uxyv$ for some $x,y\in V(C)$. 
As $C$ is induced, $u$ and $v$ are non-adjacent. Hence, there exists a vertex $w$ not on $C$ that is adjacent to $u$ and $v$ (as $G$ has diameter~$2$). Then the subgraph of $G$ induced by $\{u,v,w,x,y\}$ contains a $C_3$ or an induced $C_5$,
contradicting the minimality of $C$. \qed
\end{proof}

\section{The Propagation Algorithm and Three Results}\label{s-pro}

We present our initial propagation algorithm, which is based on a number of (well-known) propagation rules; we illustrate Rules~\ref{r-diamond} and~\ref{r-bull} in Figures~\ref{f-diamond} and~\ref{f-bull}, respectively.

\begin{enumerate}[\bf Rule 1.]
	\item \label{r-empty} {\bf(no empty lists)} If $L(u)=\emptyset$ for some $u\in V$, then return \texttt{no}.\\[-12pt]
		\item \label{r-2sat} {\bf(not only lists of size~2)} If $|L(u)|\leq 2$ for every $u\in V$, then apply Theorem~\ref{t-2sat}.\\[-12pt]
	\item \label{r-one2} {\bf(single colour propagation)} If $u$ and $v$ are adjacent, $|L(u)|=1$, and $L(u)\subseteq L(v)$, then set $L(v) := L(v) \setminus L(u)$.\\[-12pt]
\item \label{r-diamond} {\bf(diamond colour propagation)} If $u$ and $v$ are adjacent and share two common non-adjacent neighbours $x$ and $y$ with 
$|L(x)|=|L(y)|=2$ and
$L(x)\neq L(y)$, then set $L(x):=L(x)\cap L(y)$ and $L(y):=L(x)\cap L(y)$ (so $L(x)$ and $L(y)$ get size $1$).\\[-12pt]
\item \label{r-bull} {\bf(bull colour propagation)} If $u$ and $v$ are the two degree-$1$ vertices of an induced bull $B$ of $G$ and $L(u)=L(v)=\{i\}$ for some $i\in \{1,2,3\}$ and moreover $L(w)\neq \{i\}$ for the degree-$2$ vertex $w$ of $B$, then set $L(w):=L(w)\cap \{i\}$.\\[-12pt]
\end{enumerate}

\begin{figure} [h]
	\resizebox{7.1cm}{!}{
	\begin{tikzpicture}[main_node/.style={circle,draw,minimum size=1cm,inner sep=3pt]}]
	\node[main_node](u) at (0,0){\LARGE$u$};
	\node[main_node](v) at (4,0){\LARGE$v$};
	\node[main_node](x) at (2,2){\LARGE$x$};
	\node[main_node](y) at (2,-2){\LARGE$y$};
	\draw[thick](u)--(v);
	\draw[thick](u)--(x);
	\draw[thick](u)--(y);
	\draw[thick](v)--(x);
	\draw[thick](v)--(y);
	\node(listx) at (3,3){\LARGE $\{i,j\}$};
	\node(listy) at (3,-3){\LARGE$\{i,k\}$};
	\hspace*{2cm}
	\node[main_node](u2) at (6,0){\LARGE$u$};
	\node[main_node](v2) at (10,0){\LARGE$v$};
	\node[main_node](x2) at (8,2){\LARGE$x$};
	\node[main_node](y2) at (8,-2){\LARGE$y$};
	\draw[thick](u2)--(v2);
	\draw[thick](u2)--(x2);
	\draw[thick](u2)--(y2);
	\draw[thick](v2)--(x2);
	\draw[thick](v2)--(y2);
	\node(listx2) at (9,3){ \LARGE$\{i,j\} \cap \{i,k\}=\{i\}$};
	\node(listy2) at (9,-3){\LARGE$\{i,k\} \cap \{i,j\}=\{i\}$};
	
	\end{tikzpicture}
}
	\caption{Left: A diamond graph before applying Rule~\ref{r-diamond}. Right: After applying Rule~\ref{r-diamond}.}\label{f-diamond}
	
\end{figure}
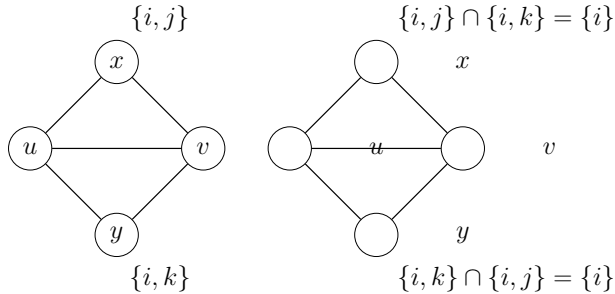
\begin{figure} [h]
	\resizebox{8.9cm}{!} {
	\begin{tikzpicture}[main_node/.style={circle,draw,minimum size=1cm,inner sep=3pt]}]
	\node[main_node](w) at (0,0){\LARGE$w$};
	\node[main_node](x) at (-2,-2){};
	\node[main_node](y) at (2,-2){};
	\node[main_node](u) at (-2,-4){\LARGE$u$};
	\node[main_node](v) at (2,-4){\LARGE$v$};
	\draw[thick](u)--(x)--(y)--(v);
	\draw[thick](x)--(w);
	\draw[thick](y)--(w);
	\node(listu) at (-3,-4){\LARGE$\{i\}$};
	\node(listv) at (1,-4){\LARGE$\{i\}$};
	\node(listw) at (0,1){\LARGE$L(w)$};
	\hspace*{2cm}
	\node[main_node](w1) at (6,0){\LARGE$w$};
	\node[main_node](x1) at (4,-2){};
	\node[main_node](y1) at (8,-2){};
	\node[main_node](u1) at (4,-4){\LARGE$u$};
	\node[main_node](v1) at (8,-4){\LARGE$v$};
	\draw[thick](u1)--(x1)--(y1)--(v1);
	\draw[thick](x1)--(w1);
	\draw[thick](y1)--(w1);
	\node(listu1) at (5,-4){\LARGE$\{i\}$};
	\node(listv1) at (9,-4){\LARGE$\{i\}$};
	\node(listw1) at (8,1){\LARGE$L(w) \cap \{i\}\;\;$ (so $L(w):=\{i\}$ or $L(w):=\emptyset$)};
	
	\end{tikzpicture}
}
	\caption{Left: A bull graph before applying Rule~\ref{r-bull}. Right: After applying Rule~\ref{r-bull}.}\label{f-bull}
\end{figure}
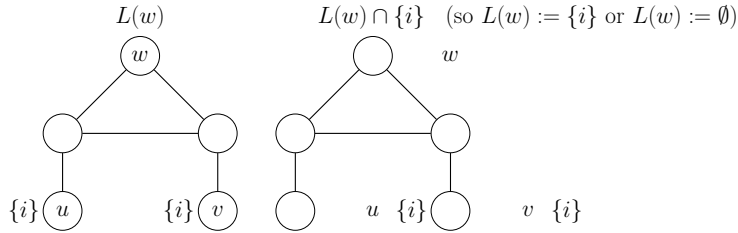

\noindent
We say that a propagation rule is {\it safe} if the new instance is a yes-instance of {\sc List 3-Colouring} if and only if the original instance is so.
We make the following observation, which is straightforward (see also~\cite{KMMNPS18}).

\begin{lemma}\label{l-easy}
Each of the Rules~\ref{r-empty}--\ref{r-bull} is safe and can be applied in polynomial time.
\end{lemma}
\noindent
Consider again an instance $(G,L)$.
Let $N_0$ be a subset of $V(G)$ that has size at most some constant. 
Assume that $G[N_0]$ has a colouring $c$ that respects the restriction of $L$ to $N_0$.
We say that $c$ is an {\it $L$-promising $N_0$-precolouring} of $G$.

In our algorithms we first determine a set $N_0$ of constant size and consider every $L$-promising $N_0$-precolouring of $G$. That is, we modify 
$L$ into a list assignment~$L_c$ with $L_c(u)=\{c(u)\}$ (where $c(u)\in L(u)$) for every $u\in N_0$ and $L_c(u)=L(u)$ for every $u\in V(G)\setminus N_0$).
We then apply Rules~\ref{r-empty}--\ref{r-bull}  on $(G,L_c)$ {\it exhaustively}, that is, until none of the rules can be applied anymore. This is the {\it propagation algorithm} and we say that it did a {\it full $c$-propagation}.
The propagation algorithm may output  {\tt yes} and {\tt no} (when applying Rules~\ref{r-empty} or~\ref{r-2sat}); else
it will output {\tt unknown}.

If the algorithm returns {\tt yes}, then $(G,L)$ is a yes-instance of {\sc List $3$-Colouring} by Lemma~\ref{l-easy}.
If it returns {\tt no}, then $(G,L)$ has no $L$-respecting colouring coinciding with~$c$ on~$N_0$, again by Lemma~\ref{l-easy}. If the algorithm returns 
{\tt unknown}, then $(G,L)$ may still have an $L$-respecting colouring that coincides with $c$ on $N_0$. In that case the propagation algorithm did not apply Rule~\ref{r-empty} or~\ref{r-2sat}. Hence, it modified $L_c$ into a list assignment~$L'_c$ of~$G$ such that $L'_c(u)\neq \emptyset$ for every $u\in V(G)$ and at least one vertex $v$ of $G$ still has a list $L'_c(v)$ of size~$3$, that is, $L'_c(v)=\{1,2,3\}$. We say that $L'_c$ (if it exists) is the {\it $c$-propagated list assignment} of $G$.

After performing a full $c$-propagation for every $L$-promising $N_0$-precolouring~$c$ of $G$ we say that we performed a {\it full $N_0$-propagation}.
We say that $(G,L)$ is {\it $N_0$-terminal} if after the full $N_0$-propagation
one of the following cases hold:

\begin{enumerate}
\item for some $L$-promising $N_0$-precolouring, the propagation algorithm returned~{\tt yes};
\item for every $L$-promising $N_0$-precolouring, the propagation algorithm returned~{\tt no}.
\end{enumerate}

\noindent
Note that if $(G,L)$ is $N_0$-terminal for some set $N_0$, then we have solved {\sc List $3$-Colouring} on instance $(G,L)$.
The next lemma formalizes our approach.

\begin{lemma}\label{l-easy2}
Let $(G,L)$ be an instance of {\sc List $3$-Colouring}. Let $N_0$ be a subset of $V(G)$ of constant size.
Performing a full $N_0$-propagation takes polynomial time. Moreover, if $(G,L)$ is $N_0$-terminal, then we have solved {\sc List $3$-Colouring} on instance $(G,L)$.
\end{lemma}

\begin{proof}
The first part of the lemma follows from the facts that (i) each application of each rule is safe and takes polynomial time by Lemma~\ref{l-easy}; (ii) if a rule does not return a {\tt yes}-answer or {\tt no-answer}, then it reduces the list size of at least one vertex and the latter can happen at most $3|V|$ times; and (iii) the number of $L$-promising $N_0$-precolourings of $G$ is at most $3^{|N_0|}$, which is a constant as $N_0$ has constant size. The second part of the lemma follows from the definition of a full $N_0$-propagation and Lemma~\ref{l-easy}. \qed
\end{proof}

\noindent
We now prove our first three results on {\sc List $3$-Colouring} for diameter-$2$ graphs. The first result generalizes a corresponding result for {\sc $3$-Colouring} in~\cite{MPS19}.	
	
\begin{theorem}\label{t-c5}
{\sc List $3$-Colouring} can be solved in polynomial time for $C_5$-free graphs of diameter at most~$2$.
\end{theorem}

\begin{proof}
Let $G=(V,E)$ be a $C_5$-free graph of diameter~$2$ with a list $3$-assignment~$L$.
We first check in polynomial time if $G$ is bipartite.
Suppose that we find that $G$ is bipartite, say with partition classes $A$ and $B$. As $G$ has diameter~$2$, we find that $G$ must be complete bipartite. This implies that either $A$ or $B$ must be monochromatic. For each $i\in \bigcap_{u\in A}L(u)$ (which might be empty) we set $L(u)=\{i\}$ for every $u\in A$ and $L(v):=L(v)\setminus \{i\}$ for every $i\in B$ and apply Theorem~\ref{t-2sat}. If we do not find a colouring respecting $L$, then we reverse the role of $A$ and $B$ and perform the same step.

Now suppose that we find that $G$ is not bipartite. If $G$ contains a $K_4$, then $G$ is not $3$-colourable, and hence $(G,L)$ is a no-instance of {\sc List $3$-Colouring}. We can check this in $O(|V|^4)$ time. From now on we assume that $G$ is $K_4$-free and non-bipartite. The latter implies that $G$ must have a triangle or an induced $C_5$, due to Lemma~\ref{l-cycle}.
As $G$ is $C_5$-free, it follows that $G$ has at least one triangle.

\begin{figure} [h]
		\resizebox{12.5cm}{!} {
			\begin{tikzpicture}[main_node/.style={circle,draw,minimum size=1cm,inner sep=5pt]}]
			\node[main_node](x1) at (0,0){\LARGE$x_1$};
			\node[main_node](x2) at (4,0){\LARGE$x_2$};
			\node[main_node](x3) at (8,0){\LARGE$x_3$};

			\draw[thick] (x1)--(x2)--(x3);
			\draw[thick] (x1) to[out=30, in=150] (x3);
			
			\draw[thick, dash dot] (-4,-1)--(10,-1);
			\node(n0) at (-3,0) {\huge$N_0$};
			\node(n1) at (-3,-2) {\huge $N_1$};
			\node(n2) at (-3,-4) {\huge$N_2$};
			\node[main_node](y1) at (0,-2){\LARGE$y_1$};
			\node[main_node](y2) at (4,-2){\LARGE$y_2$};
			\node[main_node](y3) at (8,-2){\LARGE$y_3$};
			\node[main_node](u) at (0,-4){\LARGE$u$};
			\draw[thick](x1)--(y1);
			\draw[thick](x2)--(y2);
			\draw[thick](x3)--(y3);
			\draw[thick](y1)--(u);
			\draw[thick](y2)--(u);
			\draw[thick](y3)--(u);

			\draw[thick, dash dot] (-4,-3)--(10,-3);

			\draw[thick, dash dot] (-2,2)--(-2,-5);
			
			\node[main_node](x4) at (12,0){\LARGE$x_1$};
			\node[main_node](x5) at (16,0){\LARGE$x_2$};
			\node[main_node](x6) at (20,0){\LARGE$x_3$};

			\draw[thick] (x4)--(x5)--(x6);
			\draw[thick] (x4) to[out=30, in=150] (x6);
			
			\draw[thick, dash dot] (10,-1)--(22,-1);
			
			\node[main_node](y4) at (12,-2){\LARGE$y_1$};
			\node[main_node](y5) at (16,-2){\LARGE$y_2$};
			\node[main_node](u2) at (12,-4){\LARGE$u$};
			\draw[thick](x4)--(y4);
			\draw[thick](x5)--(y5);
			\draw[thick](x6)--(y5);
			\draw[thick](y4)--(y5);
			\draw[thick](y4)--(u2);
			\draw[thick](y5)--(u2);

			\draw[thick, dash dot] (10,-3)--(22,-3);

			\draw[thick, dash dot] (10,2)--(10,-5);

			\end{tikzpicture}
		}
		\caption{Left: Examining the situation in the proof of Theorem~\ref{t-c5} where a vertex $u\in N_2$ does not belong to $T$; we show that $y_1$, $y_2$,  $y_3$ and $u$ either form a $K_4$ or we would find an induced $C_5$ (both of these cases are not possible). Right: A situation where $u\in T$. }\label{f-left}
	\end{figure}
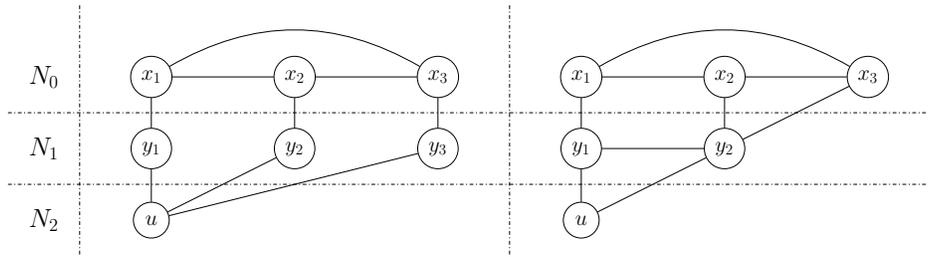
	
Let $C$ be a triangle in $G$. We write $N_0=V(C)=\{x_1,x_2,x_3\}$, $N_1=N(V(C))$ and $N_2=V(G)\setminus (N_0\cup N_1)$.
As $N_0$ has size~3, we can apply a full $N_0$-propagation in polynomial time by Lemma~\ref{l-easy2}. By the same lemma we are done if we can prove that $(G,L)$ is $N_0$-terminal. We prove this claim below after first showing a structural result.

As $G$ has diameter~$2$, for every $i\in \{1,2,3\}$, it holds that every vertex in $N_2$ has a neighbour in $N_1$ that is adjacent to $x_i$. 
Now let $T$ consist of all vertices of $N_2$ that have a neighbour in $N_1$ that is adjacent to exactly two vertices of $N_0$. 

\medskip
\noindent
{\it Claim 1. $N_2=T$.}

\medskip
\noindent
We prove Claim~1 as follows. 
Let $u\in N_2$. For contradiction, assume $u\notin T$.
If $u$ has a neighbour $y\in N_1$ adjacent to every $x_i$, then $G$ contains a $K_4$, a contradiction.
Hence, as $u\notin T$, we find that $u$ must have three distinct neighbours $y_1,y_2,y_3$, such that for $i\in \{1,2,3\}$, it holds that $N(y_i)\cap N_0=\{x_i\}$. If $\{y_1,y_2,y_3\}$ is a clique, then $G$ has a $K_4$ on vertices $u,y_1,y_2,y_3$, a contradiction. 
Hence, we may assume without loss of generality that $y_1$ and $y_2$ are non-adjacent. However, then $\{u,y_1,x_1,x_2,y_2\}$ induces a $C_5$ in $G$, another contradiction. See also Figure~\ref{f-left}. We conclude that $T=N_2$. This proves Claim~1.

\medskip
\noindent
Now, for contradiction, assume that $(G,L)$ is not $N_0$-terminal. Then there must exist an $L$-promising $N_0$-precolouring $c$
for which we obtain the $c$-propagated list assignment $L'_c$. By definition of $L'_c$ we find that $G$ contains 
a vertex $u$ with $L'_c(u)=\{1,2,3\}$. Then $u\notin N_0$, as every $v\in N_0$ has $L'_c(v)=\{c(v)\}$. 
Moreover, $u\notin N_1$, as vertices in $N_1$ have a list of size at most~$2$ after applying Rule~\ref{r-one2}.
Hence, we find that $u\in N_2$. As $N_2=T$ by Claim~1, we find that $u\in T$. From the definition of $T$ it follows that $u$ has a neighbour $v\in N_1$ with two neighbours in~$N_0$. By Rule~\ref{r-one2}, we find that $|L_c(v)|=1$. By the same rule, this implies that $|L'_c(u)|\leq 2$, a contradiction.
We conclude that $(G,L)$ is $N_0$-terminal. \qed
\end{proof}

\begin{figure} [h]
		\resizebox{12.5cm}{!} {
			\begin{tikzpicture}[main_node/.style={circle,draw,minimum size=1cm,inner sep=3pt]}]
			\node[main_node](x1) at (0,0){\LARGE$x_1$};
			\node[main_node](x2) at (4,0){\LARGE$x_2$};
			\node[main_node](x3) at (8,0){\LARGE$x_3$};
			\node[main_node](x4) at (12,0){\LARGE$x_4$};
			\node[main_node](x5) at (16,0){\LARGE$x_5$};	
			
			\draw [thick] (x1)--(x2)--(x3)--(x4)--(x5);
			\draw [thick](x1) to[out=30, in=150] (x5);
			\node(n0) at (-3,0) {\huge $N_0$};
			\node(n1) at (-3,-2) {\huge $N_1$};
			\node(n2) at (-3,-4) {\huge$N_2$};
			\draw[thick, dash dot] (-4,-1)--(10,-1);
			
			\node[main_node](y) at (0,-2){\LARGE$y$};
			\node[main_node](z) at (4,-2){\LARGE$z$};
			\node[main_node](v) at (0,-4){\LARGE$v$};
			\node[main_node](w) at (8,-2){\LARGE$w$};
			\draw[thick](x1)--(y)--(v);
			\draw[thick](x3)--(z)--(v);
			\draw[thick](y)--(z);
			\draw[thick](x4)--(w)--(v);
			\draw[thick](w) to[out=150, in=30] (y);

			\draw[thick, dash dot] (-4,-3)--(10,-3);

			\draw[thick, dash dot] (-2,2)--(-2,-5);
			
			\end{tikzpicture}
		}
			\caption{The situation in the proof of Theorem~\ref{t-c6}, which is similar to the situation in the proof of Theorem~\ref{t-c4c7}.}\label{f-both}
			\end{figure}
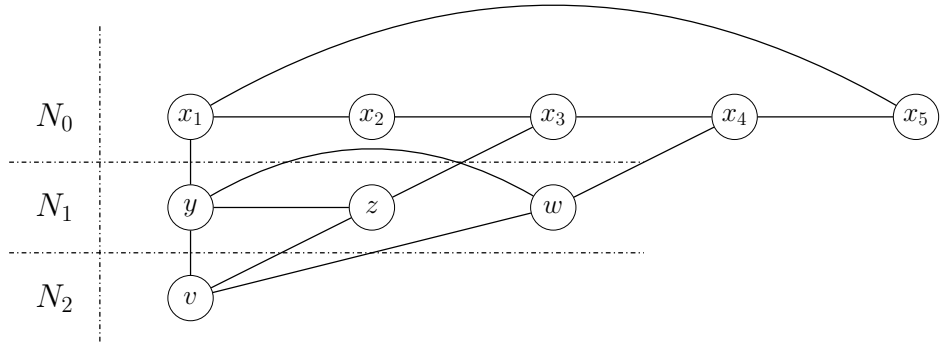

\begin{theorem}\label{t-c6}
{\sc List $3$-Colouring} can be solved in polynomial time for $C_6$-free graphs of diameter at most $2$.
\end{theorem}

\begin{proof}
Let $G=(V,E)$ be a $C_6$-free graph of diameter~$2$ with a list $3$-assignment~$L$. If $G$ is $C_5$-free, then we apply Theorem~\ref{t-c5}. 
 If $G$ contains a $K_4$, then $G$ is not $3$-colourable and hence, $(G,L)$ is a no-instance of {\sc List $3$-Colouring}.
 We check these properties in polynomial time. So, from now on, we assume that $G$ is a $K_4$-free graph that contains an induced $5$-vertex cycle~$C$, say with vertex set $N_0=\{x_1,\ldots,x_5\}$ in this order. Let $N_1$ be the set of vertices that do not belong to $C$ but that are adjacent to at least one vertex of $C$. Let $N_2=V\setminus (N_0\cup N_1)$  be the set of remaining vertices.

As $N_0$ has size~$5$, we can apply a full $N_0$-propagation in polynomial time by Lemma~\ref{l-easy2}. By the same lemma we are done if we can prove that $(G,L)$ is $N_0$-terminal. We prove this claim below.

For contradiction, assume that $(G,L)$ is not $N_0$-terminal. Then there must exist an $L$-promising $N_0$-precolouring $c$
for which we obtain the $c$-propagated list assignment $L'_c$. By definition of $L'_c$ we find that $G$ contains a 
vertex $v$ with $L'_c(v)=\{1,2,3\}$. Then $v\notin N_0$, as every $u\in N_0$ has $L'_c(u)=\{c(u)\}$. 
Moreover, $v\notin N_1$, as vertices in $N_1$ have a list of size at most~$2$ after applying Rule~\ref{r-one2}.
Hence, we find that $v\in N_2$. 

We first note that some colour of $\{1,2,3\}$ appears exactly once on $N_0$, as $|N_0|=5$. Hence, we may assume without loss of generality that $c(x_1)=1$ and that $c(x_i)\in \{2,3\}$ for every $i\in \{2,3,4,5\}$.

As $G$ has diameter $2$, there exists a vertex $y\in N_1$ that is adjacent to $x_1$ and~$v$. As $L_c'(v)=\{1,2,3\}$ and $c(x_1)=1$, we find that $L_c'(y)=\{2,3\}$. As $c(x_i)\in \{2,3\}$ for every $i\in \{2,3,4,5\}$, the
 latter means that $y$ is not adjacent to any $x_i$ with $i\in \{2,3,4,5\}$. Hence, as $G$ has diameter~$2$, there exists a vertex $z\in N_1$ with $z\neq y$, such that $z$ is adjacent to $x_3$ and $v$. We assume without loss of generality that $c(x_3)=3$ and thus $c(x_2)=c(x_4)=2$ and thus $c(x_5)=3$.  As $L_c'(v)=\{1,2,3\}$ and $c(x_3)=3$, we find that $L_c'(z)=\{1,2\}$. 
Hence, $z$ is not adjacent to any vertex of $\{x_1,x_2,x_4\}$.
Now the set $\{x_1,x_2,x_3,z,v,y\}$ forms a cycle on six vertices. As $G$ is $C_6$-free, this cycle cannot be induced. Hence, the above implies that $y$ and $z$ must be adjacent; see also Figure~\ref{f-both}.

As $G$ has diameter~$2$, there exists a vertex $w\in N_1$ that is adjacent to $x_4$ and $v$. As both $y$ and $z$ are not adjacent to $x_4$, we find that $w\notin \{y,z\}$.
As $L'_c(v)=\{1,2,3\}$ and $c(x_4)=2$, we find that $L_c'(w)=\{1,3\}$.
As $c(x_1)=1$ and $c(x_3)=c(x_5)=3$, the latter
 implies that $w$ is not adjacent to any vertex of $\{x_1,x_3,x_5\}$. Consequently, $w$ must be adjacent to $y$, as otherwise the $6$-vertex cycle with vertex set $\{x_1,x_5,x_4,w,v,y\}$ would be induced, contradicting the $C_6$-freeness of $G$. We refer again to Figure~\ref{f-both} for a display of the situation.

If $w$ and $z$ are adjacent, then $\{v,w,y,z\}$ induces a $K_4$, contradicting the $K_4$-freeness of $G$. 
Hence, $w$ and $z$ are not adjacent.
Then $\{v,w,y,z\}$ induces a diamond, in which $w$ and $z$ are the two non-adjacent vertices. 
However, as $L_c'(w)=\{1,3\}$ and $L_c'(z)=\{1,2\}$, our algorithm would have applied Rule~\ref{r-diamond}. This would have resulted in lists of $w$ and $z$ that are both equal to $\{1,3\}\cap \{1,2\}=\{1\}$. Hence, we obtained a contradiction and conclude that $(G,L)$ is $N_0$-terminal. \qed
\end{proof}

\noindent
Theorem~\ref{t-c4c7} is proven in a similar way as Theorem~\ref{t-c6}.

\begin{theorem}\label{t-c4c7}
{\sc List $3$-Colouring} can be solved in polynomial time for $(C_4,C_7)$-free graphs of diameter~$2$.
\end{theorem}
\begin{proof}
Let $G=(V,E)$ be a $C_4$-free graph of diameter~$2$ with a list $3$-assignment~$L$. If $G$ is $C_5$-free, then we apply Theorem~\ref{t-c5}. Hence we may assume that $G$ contains an induced $5$-vertex cycle~$C$, say with vertex set $N_0=\{x_1,\ldots,x_5\}$ in this order. As before, we let $N_1$ be the set of vertices that do not belong to $C$ but that are adjacent to at least one vertex of $C$. We also let $N_2=V\setminus (N_0\cup N_1)$  denote the set of remaining vertices again.

As $N_0$ has size~$5$, we can apply a full $N_0$-propagation in polynomial time by Lemma~\ref{l-easy2}. By the same lemma we are done if we can prove that $(G,L)$ is $N_0$-terminal. We prove this claim in exactly the same way in which we proved a similar claim in 
the proof of Theorem~\ref{t-c6} except for the following differences:
\begin{enumerate}
\item instead of using the $6$-vertex set $\{x_1,x_2,x_3,z,v,y\}$ we use the $7$-vertex set 
$\{x_1,x_5,x_4,x_3,z,v,y\}$ after observing that $z$ cannot be adjacent to $x_5$ due to the $C_4$-freeness of $G$, and
\item instead of using the $6$-vertex set $\{x_1,x_5,x_4,w,v,y\}$ we use the $7$-vertex set $\{x_1,x_2,x_3,x_4,w,v,y\}$ after observing that $w$ cannot be adjacent to $x_2$, again due to the $C_4$-freeness of $G$.
\end{enumerate}
We refer again to Figure~\ref{f-both} for a display of the situation. \qed
\end{proof}

\section{The Extended Propagation Algorithm and Two Results}\label{s-pro2}

For our next two results, we need a more sophisticated method. Let $(G,L)$ be an instance of {\sc List $3$-Colouring}. 
Let $p$ be some positive constant. We consider each set $N_0\subseteq V(G)$ of size at most~$p$ and perform a full $N_0$-propagation. 
Afterwards we say that we performed a {\it full $p$-propagation}.
We say that $(G,L)$ is {\it $p$-terminal} if after the full $p$-propagation one of the following cases hold:

\begin{enumerate}
\item for some $N_0\subseteq V(G)$ with $|N_0|\leq c$, there is an $L$-promising $N_0$-precolouring~$c$, such that the propagation algorithm returns {\tt yes}; or
\item for every set $N_0\subseteq V(G)$ with $|N_0|\leq c$ and every $L$-promising $N_0$-precolouring~$c$, the propagation algorithm returns {\tt no}.
\end{enumerate}

\noindent
We can now prove the following lemma.

\begin{lemma}\label{l-easy3}
Let $(G,L)$ be an instance of {\sc List $3$-Colouring} and $p\geq 1$ be some constant.
Performing a full $p$-propagation takes polynomial time. Moreover, if $(G,L)$ is $p$-terminal, then we have solved {\sc List $3$-Colouring} on instance $(G,L)$.
\end{lemma}

\begin{proof}
For every set $N_0\subseteq V(G)$, a full $N_0$-propagation takes polynomial time by Lemma~\ref{l-easy2}.
Then the first statement of the lemma follows from this observation and the fact that we need to perform $O(n^p)$ full $N_0$-propagations, which is a polynomial number, as $p$ is a constant.

Now suppose that $(G,L)$ is $p$-terminal. First assume that for some $N_0\subseteq V(G)$ with $|N_0|\leq c$, there exists an $L$-promising $N_0$-precolouring $c$, such that the propagation algorithm returns {\tt yes}. Then $(G,L)$ is a yes-instance due to Lemma~\ref{l-easy}.
Now assume that for every set $N_0\subseteq V(G)$ with $|N_0|\leq c$ and every  $L$-promising $N_0$-precolouring $c$, the propagation algorithm returns {\tt no}. Then $(G,L)$ is a no-instance. This follows from Lemma~\ref{l-easy} combined with the observation that if $(G,L)$ was a yes-instance, the restriction of a colouring $c$ that respects $L$ to any set $N_0$ of size at most $p$ would be an $L$-promising $N_0$-precolouring of $G$. \qed
\end{proof}

\noindent
In our next two algorithms, we perform a full $p$-propagation for some appropriate constant~$p$. If we find that an instance $(G,L)$ is $p$-terminal, then we are done by Lemma~\ref{l-easy3}. In the other case, we exploit the new information on the structure of~$G$ that we obtain from the fact that $(G,L)$ is not $p$-terminal.

\begin{theorem}\label{t-c4c8}
{\sc List $3$-Colouring} can be solved in polynomial time for $(C_4,C_8)$-free graphs of diameter~$2$.
\end{theorem}

\begin{proof}
Let $G=(V,E)$ be a $(C_4,C_8)$-free graph of diameter~$2$ with a list $3$-assignment $L$. If $G$ is $C_6$-free, then we apply Theorem~\ref{t-c6}. 
 If $G$ contains a $K_4$, then $G$ is not $3$-colourable and hence, $(G,L)$ is a no-instance of {\sc List $3$-Colouring}. We check these properties in polynomial time. So, from now on, we assume that $G$ is a $K_4$-free graph that contains at least one induced cycle on six vertices.
 
We set $p=6$ and perform a full $p$-propagation. This takes polynomial time by Lemma~\ref{l-easy}. By the same lemma, we have solved {\sc List $3$-Colouring} on $(G,L)$ if $(G,L)$ is $p$-terminal. Suppose we find that $(G,L)$ is not $p$-terminal.

We first prove the following claim.

\medskip
\noindent
{\it Claim 1. For each induced $6$-vertex cycle $C$, the propagation algorithm returned {\tt no} for every $V(C)$-promising colouring $c$ that assigns the same colour $i$ on two vertices of $C$ that have a common neighbour on $C$.}

\medskip
\noindent
We prove Claim~1 as follows.
Consider an induced $6$-vertex cycle~$C$, say with vertex set $N_0=\{x_1,\ldots,x_6\}$ in this order. Let $N_1$ be the set of vertices that do not belong to $C$ but that are adjacent to at least one vertex of $C$. Let $N_2=V\setminus (N_0\cup N_1)$  be the set of remaining vertices.
For contradiction, let $c$ be a $V(C)$-promising colouring that assigns two vertices of $C$ with a common neighbour on $C$ the same colour, say $c(x_1)=1$ and $c(x_3)=1$, such that a full $c$-propagation does not yield a {\tt no} output. As $(G,L)$ is not $p$-terminal, this means
that we obtained the $c$-propagated list assignment $L'_c$. By definition of $L'_c$ we find that $G$ contains  
a vertex $v$ with $L'_c(v)=\{1,2,3\}$. Then $v\notin N_0$, as every $u\in N_0$ has $L'_c(u)=\{c(u)\}$. 
Moreover, $v\notin N_1$, as vertices in $N_1$ have a list of size at most~$2$ after applying Rule~\ref{r-one2}.
Hence, we find that $v\in N_2$. 

As $G$ has diameter~$2$, there exist a vertex $y\in N_1$ that is adjacent to both~$v$ and~$x_1$.
As $c(x_1)=1$, we find that $c(x_2)\in \{2,3\}$ and $c(x_6)\in \{2,3\}$.
As $c(x_3)=1$, we find that $c(x_4)\in \{2,3\}$. Hence, $y$ is not adjacent to any vertex of $\{x_2,x_4,x_6\}$; 
otherwise $y$ would have a list of size~$1$ due to Rule~\ref{r-one2}, and by the same rule, $v$ would have a list of size~$2$.
We note that $y$ is not adjacent to $x_3$  
or $x_5$ either, as otherwise $\{x_1,x_2,x_3,y\}$ or $\{x_1,x_6,x_5,y\}$ induces a $C_4$, contradicting the $C_4$-freeness of $G$.

As $G$ has diameter~$2$ and $yx_3\notin E$, there exists a vertex $y'\in N_1\setminus \{y\}$ that is adjacent to both $v$ and $x_3$. By the same arguments as above, $y'$ is not adjacent to any vertex of $\{x_1,x_2,x_4,x_5,x_6\}$.
If $y$ and $y'$ are adjacent, then $v$ would have list $\{1\}$ due to Rule~\ref{r-bull}.
Hence $y$ and $y'$ are not adjacent. However, we now find that $\{x_1,y,v,y',x_3,x_4,x_5,x_6\}$ induces a $C_8$, contradicting the $C_8$-freeness of $G$; see also Figure~\ref{f-c8}. This proves Claim~1.

\begin{figure}
	\resizebox{9.5cm}{!}{
	\begin{tikzpicture}[main_node/.style={circle,draw,minimum size=1cm,inner sep=3pt]}]
	
	\node[main_node](v) at (0,0){\LARGE$v$};
	\node[main_node](y_1) at (0,2){\LARGE$y$};
	\node[main_node](y_2) at (2,2){\LARGE$y^{\prime}$};
	\node[main_node](x_1) at (0,4){\LARGE$x_1$};
	\node[main_node](x_2) at (2,4){\LARGE$x_2$};
	\node[main_node](x_3) at (4,4){\LARGE$x_3$};
	\node[main_node](x_4) at (6,4){\LARGE$x_4$};
	\node[main_node](x_5) at (8,4){\LARGE$x_5$};
	\node[main_node](x_6) at (10,4){\LARGE$x_6$};
	
	\draw[thick](x_1)--(x_2)--(x_3)--(x_4)--(x_5)--(x_6);
	\draw[thick](v)--(y_1);
	\draw[thick](v)--(y_2);
	\draw[thick](y_1)--(x_1);
	\draw[thick](y_2)--(x_3);
	\draw[thick](x_1) to[out=60, in=120] (x_6);
	\draw[thick, dash dot] (-4,1)--(10,1);
	\draw[thick, dash dot] (-4,3)--(10,3);
	\draw[thick, dash dot] (-3, 5)--(-3, -1 );
	\node(x_1list) at (0,5){\LARGE$\{1\}$};
	\node(x_3list) at (4,5) {\LARGE$\{1\}$};
	\node(x_2list) at (2,5) {\LARGE$ \subset \{2,3\}$};
	\node(x_4list) at (5.8,5) {\LARGE$ \subset \{2,3\}$};
		\node(x_5list) at (8,5) {\LARGE$ \subset \{1,2,3\}$};
	\node(x_6list) at (11,5) {\LARGE$\subset \{2,3\}$};
	\end{tikzpicture}
}
	\caption{The situation that is described in Claim 1 in the proof of Theorem~\ref{t-c4c8}: the set $\{x_1,y,v,y',x_3,x_4,x_5,x_6\}$ induces a $C_8$, which is not possible.}\label{f-c8}
	\end{figure}
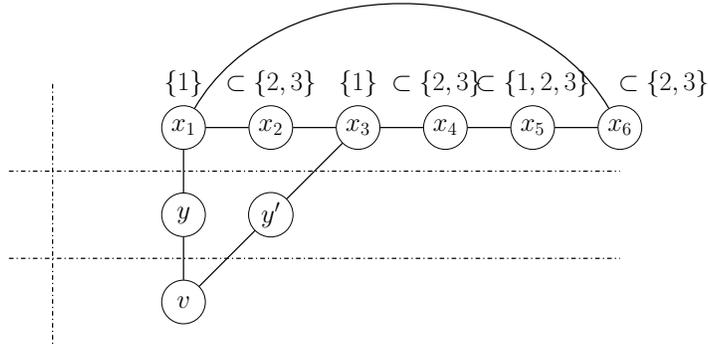

\medskip
\noindent
Due to Claim~1, we know that if $G$ has a colouring~$c$ respecting $L$, then any such colouring~$c$ gives a different colour to every two non-adjacent vertices that are of distance~$2$ on some induced $6$-vertex cycle. Hence, we can safely use the following new rule. To explain this, $x_5$ cannot get the same colour of both $x_1$ and $x_3$, which are both of distance~$2$ from $x_5$ on an induced $C_6$, thus $x_5$ must get the remaining colour, which is the colour of $x_2$. Moreover, an application of the new rule takes polynomial time.
Note that we must also have that $L(x_4)=L(x_1)$ and $L(x_6)=L(x_3)$ but this will be irrelevant for our purposes.

\begin{enumerate}[\bf Rule-C6.]
	\item \label{r-c6} {\bf(${\mathbf {C_6}}$ colour propagation)} 
	Let $C$ be an induced cycle on six vertices $x_1,x_2,\ldots,x_6$ in that order. If $|L(x_1)|=|L(x_2)|=|L(x_3)|=1$, $L(\{x_1,x_2,x_3\})=\{1,2,3\}$ and $L(x_2)\neq L(x_5)$, then set $L(x_5):=L(x_2)\cap L(x_5)$ (so $x_5$ gets a list of size at most~$1$).\\[-12pt]
\end{enumerate}

\noindent
\begin{figure}
	\resizebox{9.5cm}{!}{
	\begin{tikzpicture}[main_node/.style={circle,draw,minimum size=1cm,inner sep=3pt]}]
	
	\node[main_node](v) at (0,0){\LARGE$v$};
	\node[main_node](y_1) at (0,2){\LARGE$y$};
	\node[main_node](y_2) at (2,2){\LARGE$y^{\prime}$};
	\node[main_node](y_3) at (5,2){\LARGE$y^{\prime \prime}$};
	\node[main_node](x_1) at (0,4){\LARGE$x_1$};
	\node[main_node](x_2) at (2,4){\LARGE$x_2$};
	\node[main_node](x_3) at (4,4){\LARGE$x_3$};
	\node[main_node](x_4) at (6,4){\LARGE$x_4$};
	\node[main_node](x_5) at (8,4){\LARGE$x_5$};
	\node[main_node](x_6) at (10,4){\LARGE$x_6$};
	
	\draw[thick](x_1)--(x_2)--(x_3)--(x_4)--(x_5)--(x_6);
	\draw[thick](v)--(y_1);
	\draw[thick](v)--(y_2);
	\draw[thick](y_1)--(x_1);
	\draw[thick](y_2)--(x_3);
	\draw[thick](x_1) to[out=60, in=120] (x_6);
	\draw[thick](y_1)--(y_2)--(y_3);
	\draw[thick](y_1) to[out=25, in=155] (y_3);
	\draw[thick](y_2)--(x_5);
	\draw[thick](y_3)--(v);
		\draw[thick](y_3)--(x_5);
	\draw[thick, dash dot] (-4,1)--(10,1);
	\draw[thick, dash dot] (-4,3)--(10,3);
	\draw[thick, dash dot] (-3, 5)--(-3, -1 );
	\node(x_1list) at (0,5){\LARGE$\{1\}$};
	\node(x_3list) at (4,5) {\LARGE$\{3\}$};
	\node(x_2list) at (2,5) {\LARGE$\{2\}$};
	\node(x_4list) at (6,5) {\LARGE$\{1\}$};
	\node(x_6list) at (10.5,5) {\LARGE$\{3\}$};
	\node(x_5list) at (8,5) {\LARGE$\{2\}$};
	\end{tikzpicture}
}
	\caption{The situation in the proof of Theorem~\ref{t-c4c8}, where a vertex $v\in N_2$ still has a list of three available colours after a full propagation including Rule-C6: we show that in this case $G$ contains a $K_4$, namely on vertices $v$, $y$, $y'$, $y''$, a contradiction.}\label{f-22}
\end{figure}
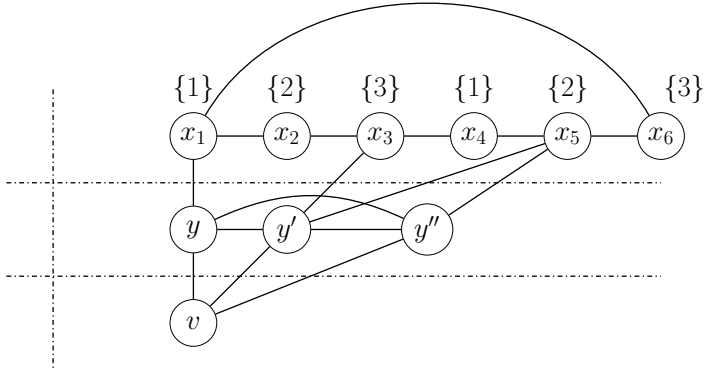
\noindent
We can now do as follows. Consider an induced $6$-vertex cycle~$C$ in $G$, say on vertices $x_1,\ldots,x_6$ in that order.
Then we may assume without loss of generality that if $G$ has a colouring $c$ that respects $L$, then $c(x_1)=1$, $c(x_2)=2$, $c(x_3)=3$, $c(x_4)=1$, $c(x_5)=2$ and $c(x_6)=3$ (otherwise we can do some permutation of the colours). See also Figure~\ref{f-22}.

We let again $N_0=\{x_1,\ldots,x_6\}$, $N_1$ be the set of vertices that do not belong to $C$ but that are adjacent to at least one vertex of $C$, and $N_2=V\setminus (N_0\cup N_1)$  be the set of remaining vertices. We  do a full $c$-propagation but now we also include the exhaustive use of Rule-C6. By combining Lemma~\ref{l-easy} with the observation that Rule-C6  runs in polynomial time and reduces the list size of at least one vertex, this takes polynomial time.
By combining the same lemma with the fact that Rule-C6 is safe (due to Claim~1) and the above observation that every $L$-respecting colouring of $G$ coincides with $c$ on $N_0$ (subject to colour permutation), we are done if we can prove that the propagation algorithm either outputs {\tt yes} or {\tt no}.

For contradiction, assume that the propagation algorithm returns {\tt unknown}. Then we obtained the $c$-propagated list assignment $L'_c$. By definition of $L'_c$ we find that $G$ contains a a vertex $v$ with $L'_c(v)=\{1,2,3\}$. Then $v\notin N_0$, as every $u\in N_0$ has $L'_c(u)=\{c(u)\}$. Moreover, $v\notin N_1$, as vertices in $N_1$ have a list of size at most~$2$ after applying Rule~\ref{r-one2}.
Hence, we find that $v\in N_2$. 

As $G$ has diameter~$2$, there exists a vertex $y\in N_1$ that is adjacent to $x_1$ and~$v$.  Hence, $y$ is not adjacent to any vertex in $\{x_2,x_3,x_5,x_6\}$;
otherwise $y$ would have a list of size~$1$ due to Rule~\ref{r-one2}, and by the same rule, $v$ would have a list of size~$2$.
As $G$ has diameter~$2$ and $yx_3\notin E$, there exists a vertex $y'\in N_1\setminus \{y\}$ that is adjacent to $x_3$ and~$v$. By the same arguments as above, $y'$ is not adjacent to any vertex in $\{x_1,x_2,x_4,x_5\}$. 
If $yy'\notin E$, then $\{x_1,x_2,x_3,y',v,y\}$ induces a $C_6$. However, in that case we would have applied Rule-C6 and $v$ would have had list $\{2\}$. Hence, we find that $y$ and $y'$ are adjacent; see also Figure~\ref{f-22}.

As $G$ has diameter~$2$, $yx_5\notin E$ and $y'x_5\notin E$, there exists a vertex $y''\in N_1\setminus \{y,y'\}$ that is adjacent to $x_5$ and~$v$. By using exactly the same arguments as above but now applied to $y''$ and to the pairs $(y,y'')$ and $(y',y'')$, respectively, we find that $y''$ is adjacent to both $y$ and $y'$.
However, now the vertices $v,y,y',y''$ induce a $K_4$, contradicting the $K_4$-freeness of $G$ (see again Figure~\ref{f-22}).
We conclude that the propagation algorithm returned either {\tt yes} or {\tt no}. \qed
\end{proof}

\begin{theorem}\label{t-c4c9}
{\sc List $3$-Colouring} can be solved in polynomial time for $(C_4,C_9)$-free graphs of diameter~$2$.
\end{theorem}

\begin{proof}
Let $G=(V,E)$ be a $(C_4,C_9)$-free graph of diameter~$2$ with a list $3$-assignment $L$. If $G$ is $C_7$-free, then we apply Theorem~\ref{t-c4c7}. 
 If $G$ contains a $K_4$, then $G$ is not $3$-colourable and hence, $(G,L)$ is a no-instance of {\sc List $3$-Colouring}. We check these properties in polynomial time. So, from now on, we assume that $G$ is a $K_4$-free graph that contains at least one induced cycle on seven vertices.
 
We set $p=7$ and perform a full $p$-propagation. This takes polynomial time by Lemma~\ref{l-easy}. By the same lemma, we have solved {\sc List $3$-Colouring} on $(G,L)$ if $(G,L)$ is $p$-terminal. Suppose we find that $(G,L)$ is not $p$-terminal.

We first prove the following claim.

\medskip
\noindent
{\it Claim 1. For each induced $7$-vertex cycle $C$, the propagation algorithm returned {\tt no} for every $L$-promising $V(C)$-colouring $c$ that 
assigns the same colour $i$ on two vertices of $C$ that have a common neighbour on $C$ and that gives every other vertex of $C$ a colour different from~$i$.}

\medskip
\noindent
We prove Claim~1 as follows.
Consider an induced $7$-vertex cycle~$C$, say with vertex set $N_0=\{x_1,\ldots,x_7\}$ in this order. Let $N_1$ be the set of vertices that do not belong to $C$ but that are adjacent to at least one vertex of $C$. Let $N_2=V\setminus (N_0\cup N_1)$  be the set of remaining vertices.
Let $c$ be an $L$-promising $V(C)$-colouring that assigns two vertices of $C$ with a common neighbour on $C$ the same colour, say $c(x_1)=1$ and $c(x_3)=1$, and moreover, that assigns every vertex $x_i$ with $i\in \{2,4,5,6,7\}$ colour
 $c(x_i)\neq 1$.

For contradiction, suppose that a full $c$-propagation does not yield a {\tt no} output. As $(G,L)$ is not $p$-terminal, this means
that we obtained the $c$-propagated list assignment $L'_c$. By definition of $L'_c$ we find that $G$ contains a 
vertex $v$ with $L'_c(v)=\{1,2,3\}$. Then $v\notin N_0$, as every $u\in N_0$ has $L'_c(u)=\{c(u)\}$. 
Moreover, $v\notin N_1$, as vertices in $N_1$ have a list of size at most~$2$ after applying Rule~\ref{r-one2}.
Hence, we find that $v\in N_2$. 

As $G$ has diameter~$2$, there exist a vertex $y\in N_1$ that is adjacent to both $v$ and $x_1$.
Then $y$ is not adjacent to any $x_i$ with $i\in \{2,4,5,6,7\}$; in that case $y$ would have a list of size~$1$ (as each $x_i$ other than $x_1$ and $x_3$ is coloured $2$ or $3$) meaning that $L_c'(v)$ would have size at most~$2$. Hence, $y$ is not adjacent to $x_3$ either, as otherwise $\{y,x_1,x_2,x_3\}$ would induce a $C_4$. As $G$ has diameter~$2$, this means that there exists a vertex $y'\in N_1$ with $y'\neq y$ such that $y'$ is adjacent to both $v$ and $x_3$. By the same arguments we used for $y'$, we find that $x_3$ is the only neighbour of $y'$ on $C$.

If $yy'$ is an edge then, by Rule~\ref{r-bull}, $v$ would have had list $\{1\}$ instead of $\{1,2,3\}$. Hence, $y$ and $y'$ are not adjacent. However, now $\{y,v,y',x_3,x_4,x_5,x_6,x_7,x_1\}$ induces a $C_9$, a contradiction; see also Figure~\ref{f-ccc1}. This proves Claim~1.

\begin{figure}[h]
	\resizebox{10.5cm}{!}{
	\begin{tikzpicture}[main_node/.style={circle,draw,minimum size=1cm,inner sep=3pt]}]
	
	\node[main_node](v) at (0,0){\LARGE$v$};
	\node[main_node](y_1) at (0,2){\LARGE$y$};
	\node[main_node](y_2) at (2,2){\LARGE$y^{\prime}$};
	\node[main_node](x_1) at (0,4){\LARGE$x_1$};
	\node[main_node](x_2) at (2,4){\LARGE$x_2$};
	\node[main_node](x_3) at (4,4){\LARGE$x_3$};
	\node[main_node](x_4) at (6,4){\LARGE$x_4$};
	\node[main_node](x_5) at (8,4){\LARGE$x_5$};
	\node[main_node](x_6) at (10,4){\LARGE$x_6$};
	\node[main_node](x_7) at (12,4){\LARGE$x_7$};
	
	\draw[thick](x_1)--(x_2)--(x_3)--(x_4)--(x_5)--(x_6)--(x_7);
	\draw[thick](v)--(y_1);
	\draw[thick](v)--(y_2);
	\draw[thick](y_1)--(x_1);
	\draw[thick](y_2)--(x_3);
	\draw[thick](x_1) to[out=60, in=120] (x_7);
	\draw[thick, dash dot] (-4,1)--(12.5,1);
	\draw[thick, dash dot] (-4,3)--(12.5,3);
	\draw[thick, dash dot] (-3, 5)--(-3, -1 );
	\node(x_1list) at (0,5){\LARGE$\{1\}$};
	\node(x_3list) at (4,5) {\LARGE$\{1\}$};
	\node(x_2list) at (2,5) {\LARGE$ \subset \{2,3\}$};
	\node(x_4list) at (6,5) {\LARGE$ \subset \{2,3\}$};
	\node(x_5list) at (8,5) {\LARGE $\subset \{2,3\}$};
	\node(x_6list) at (10,5) {\LARGE$\subset \{2,3\}$};
	\node(x_7list) at (12.8,5) {\LARGE $\subset \{2,3\}$};
	\end{tikzpicture}
}
	\caption{The situation that is described in Claim~1 in the proof of Theorem \ref{t-c4c9}. The set $\{x_1,y,v,y^{\prime}, x_3, x_4, x_5, x_6, x_7\}$ induces a $C_9$, which is not possible.}\label{f-ccc1}
	\end{figure}

\medskip
\noindent
Claim~1 tells us that if $G$ has a colouring~$c$ respecting $L$, then~$c$ only gives the same colour to two vertices $x$ and $x'$ that are of distance~$2$ on some induced $7$-vertex cycle $C$ if there is a third vertex $x''$ that is of distance~$2$ from either $x$ or $x'$ on $C$ with $c(x'')=c(x')=c(x)$.
Hence, we can safely use the following new rule, whose execution takes polynomial time (in this rule,  $c(x_1)=c(x_6)$ is not possible: view $x_1$ as $x$ and $x_6$ as $x'$ and note that $x''$ can neither be $x_3$ or $x_4$).

\begin{enumerate}[\bf Rule-C7.]
	\item \label{r-c7} {\bf(${\mathbf {C_7}}$ colour propagation)} 
	Let $C$ be an induced cycle on seven vertices $x_1,x_2,\ldots,x_7$ in that order. If $|L(x_i)|=1$ for $i\in \{1,2,3,4\}$,  $L(\{x_1,x_2,x_3\})=\{1,2,3\}$, $L(x_4)=L(x_2)$, and $L(x_1)\subseteq L(x_6)$, then set $L(x_6):=\{1,2,3\}\setminus L(x_1)$ (so $L(x_6)$ gets size at most~$2$).\\[-12pt]
\end{enumerate}

\noindent
We now consider an induced $7$-vertex cycle~$C$ in $G$, say on vertices $x_1,\ldots,x_7$ in that order.
Then either one colour appear once on $C$, or two colours appear exactly twice on $C$, with distance~$3$ from each other on $C$.
Hence, we may assume without loss of generality that if $G$ has a colouring $c$ that respects $L$, then one of the following holds for such a colouring $c$ (see also Figures~\ref{f-one} and~\ref{f-two}):
\begin{itemize}
\item [(1)]  $c(x_1)=1$, $c(x_2)=2$, $c(x_3)=3$, $c(x_4)=2$, $c(x_5)=3$, $c(x_6)=2$, $c(x_7)=3$; or
\item [(2)]  $c(x_1)=1$, $c(x_2)=2$, $c(x_3)=3$, $c(x_4)=1$, $c(x_5)=3$, $c(x_6)=2$, $c(x_7)=3$.
\end{itemize}
We let again $N_0=\{x_1,\ldots,x_7\}$, $N_1$ be the set of vertices that do not belong to~$C$ but that are adjacent to at least one vertex of~$C$, and $N_2=V\setminus (N_0\cup N_1)$  be the set of remaining vertices. We  do a full $c$-propagation but now we also include the exhaustive use of  Rule-C7. By combining Lemma~\ref{l-easy} with the observation that Rule-C7  runs in polynomial time and reduces the list size of at least one vertex, this takes polynomial time.
By combining the same lemma with the fact that Rule-C7 is safe (due to Claim~1) and the above observation that every $L$-respecting colouring of $G$ coincides with $c$ on $N_0$ (subject to colour permutation), we are done if we can prove that the propagation algorithm either outputs {\tt yes} or {\tt no}. We show that this is the case for each of the two possibilities (1) and (2) of $c$.

For contradiction, assume that the propagation algorithm returns {\tt unknown}. Then we obtained the $c$-propagated list assignment $L'_c$. By definition of $L'_c$ we find that $G$ contains a vertex $v$ with $L'_c(v)=\{1,2,3\}$. Then $v\notin N_0$, as every $u\in N_0$ has $L'_c(u)=\{c(u)\}$. Moreover, $v\notin N_1$, as vertices in $N_1$ have a list of size at most~$2$ after applying Rule~\ref{r-one2}.
Hence, we find that $v\in N_2$. We now need to distinguish between the two possibilities of $c$.

\medskip
\noindent
{\bf Case 1} $c(x_1)=1$, $c(x_2)=2$, $c(x_3)=3$, $c(x_4)=2$, $c(x_5)=3$, $c(x_6)=2$,~$c(x_7)=3$\\
As $G$ has diameter $2$, there exists a vertex $y\in N_1$ that is adjacent to $x_1$ and~$v$.  Hence, $y$ is not adjacent to any vertex in 
$\{x_2,\ldots,x_7\}$; 
otherwise $y$ would have a list of size~$1$ due to Rule~\ref{r-one2}, and by the same rule, $v$ would have a list of size~$2$.
As $G$ has diameter $2$, there exists a vertex $y'\in N_1$ that is adjacent to $x_4$ and~$v$. By the same arguments as above, $y'$ is not adjacent to any vertex of $\{x_1,x_3,x_5,x_7\}$. The latter, together with the $C_4$-freeness of $G$, implies that $y'$ is not adjacent to $x_2$ and $x_6$ either.

First suppose that $yy'\in E$. Then $\{x_1,x_7,x_6,x_5,x_4,y',y\}$ induces a $C_7$; see also Figure~\ref{f-one}.
As $c(x_1)=1$, $c(x_7)=3$, $c(x_6)=2$ and $c(x_5)=3$,
we find that $L_c(\{x_1,x_7,x_6\})=\{1,2,3\}$ and $L_c(x_5)=L_c(x_7)$. Then $1\notin L_c(y')$, as otherwise
the propagation algorithm would have applied Rule-C7. Moreover, $2\notin L_c(y')$, as otherwise the propagation algorithm would have applied Rule~\ref{r-one2}. Hence, $L_c(y')=\{3\}$. However, then $|L_c(v)|\leq 2$, again due to Rule~\ref{r-one2}, a contradiction.

Now suppose that $yy'\notin E$. Then $\{x_1,x_2,x_3,x_4,y',v,y\}$ induces a $C_7$. As $c(x_1)=1$, $c(x_2)=2$, $c(x_3)=3$, $c(x_4)=2$,
we find that $L_c(\{x_1,x_2,x_3\})=\{1,2,3\}$ and $L_c(x_4)=L_c(x_2)$. Then $1\notin L_c(v)$ due to Rule-C7. This is a contradiction, as we assumed $L_c(v)=\{1,2,3\}$.
We conclude that the propagation algorithm returned either {\tt yes} or {\tt no}. 

\begin{figure}
	\resizebox{9.5cm}{!}{
	\begin{tikzpicture}[main_node/.style={circle,draw,minimum size=1cm,inner sep=3pt]}]
	
	\node[main_node](v) at (0,0){\LARGE$v$};
	\node[main_node](y_1) at (0,2){\LARGE$y$};
	\node[main_node](y_2) at (2,2){\LARGE$y^{\prime}$};
	\node[main_node](x_1) at (0,4){\LARGE$x_1$};
	\node[main_node](x_2) at (2,4){\LARGE$x_2$};
	\node[main_node](x_3) at (4,4){\LARGE$x_3$};
	\node[main_node](x_4) at (6,4){\LARGE$x_4$};
	\node[main_node](x_5) at (8,4){\LARGE$x_5$};
	\node[main_node](x_6) at (10,4){\LARGE$x_6$};
	\node[main_node](x_7) at (12,4){\LARGE$x_7$};
	
	\draw[thick](x_1)--(x_2)--(x_3)--(x_4)--(x_5)--(x_6)--(x_7);
	\draw[thick](v)--(y_1);
	\draw[thick](v)--(y_2);
	\draw[thick](y_1)--(x_1);
	\draw[dashed](y_1)--(y_2);
	\draw[thick](y_2)--(x_4);
	\draw[thick](x_1) to[out=40, in=140] (x_7);
	\draw[thick, dash dot] (-4,1)--(12.5,1);
	\draw[thick, dash dot] (-4,3)--(12.5,3);
	\draw[thick, dash dot] (-3, 5)--(-3, -1 );
	\node(x_1list) at (0,5){\LARGE$\{1\}$};
	\node(x_3list) at (4,5) {\LARGE$\{3\}$};
	\node(x_2list) at (2.2,5) {\LARGE$\{2\}$};
	\node(x_4list) at (6,5) {\LARGE$ \{2\}$};
	\node(x_5list) at (8,5) {\LARGE $\{3\}$};
	\node(x_6list) at (9.8,5) {\LARGE$\{2\}$};
	\node(x_7list) at (12,5) {\LARGE $\{3\}$};
	\end{tikzpicture}
}
	\caption{The situation that is described in Case 1 in the proof of Theorem \ref{t-c4c9}. If the edge~$yy^{\prime}$ exists, then  $\{x_1, x_7, x_6, x_5, x_4, y^{\prime}, y\}$ induces a $C_7$ to which Rule-C7 should have been applied. Otherwise the vertices $\{x_1, x_2, x_3, x_4, y^{\prime}, v, y\}$ induce such a $C_7$.}\label{f-one}
\end{figure}
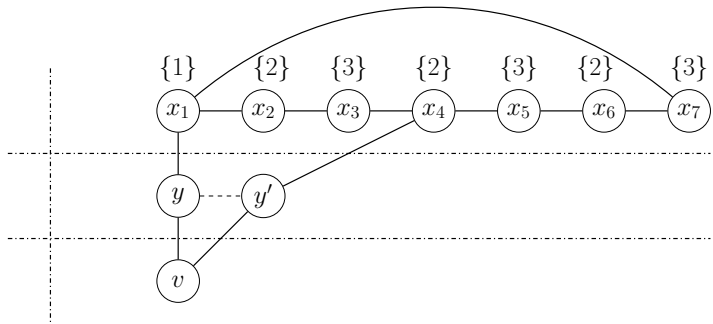

\medskip
\noindent
{\bf Case~2}~$c(x_1)=1$,~$c(x_2)=2$,~$c(x_3)=3$,~$c(x_4)=1$,~$c(x_5)=3$,~$c(x_6)=2$,~$c(x_7)=3$\\
As $G$ has diameter $2$, there is a vertex $y\in N_1$ adjacent to $x_3$ and~$v$.  Hence, $y$ is not adjacent to any vertex in
$\{x_1,x_2,x_4,x_6\}$; otherwise $y$ would have a list of size~$1$ due to Rule~\ref{r-one2}, and by the same rule, $v$ would have a list of size~$2$.
As $yx_4\notin E$, we find that $yx_5\notin E$ either; otherwise $\{y,x_3,x_4,x_5\}$ induces a $C_4$.
As $G$ has diameter $2$, this means there is a vertex $y'\in N_1\setminus \{y\}$ adjacent to $x_5$ and~$v$. By the same arguments as above, $y'$ is not adjacent to any vertex of $\{x_1,x_2,x_4,x_6\}$. As $G$ is $C_4$-free, the latter implies that $y'x_3\notin E$ and $y'x_7\notin E$.

\begin{figure}[h]
\resizebox{9.5cm}{!}{
\begin{tikzpicture}[main_node/.style={circle,draw,minimum size=1cm,inner sep=3pt]}]
\node[main_node](v) at (8,0){\Large$v$};
\node[main_node](y_1) at (6,2){\LARGE$y$};
\node[main_node](y_2) at (8,2){\LARGE$y^{\prime}$};
\node[main_node](x_1) at (0,4){\LARGE$x_1$};
\node[main_node](x_2) at (2,4){\LARGE$x_2$};
\node[main_node](x_3) at (4,4){\LARGE$x_3$};
\node[main_node](x_4) at (6,4){\LARGE$x_4$};
\node[main_node](x_5) at (8,4){\LARGE$x_5$};
\node[main_node](x_6) at (10,4){\LARGE$x_6$};
\node[main_node](x_7) at (12,4){\LARGE$x_7$};
\node[main_node](z) at (10,2){\LARGE$z$};
\draw[thick](x_1)--(x_2)--(x_3)--(x_4)--(x_5)--(x_6)--(x_7);
\draw[thick](v)--(y_1);
\draw[thick](v)--(y_2);
\draw[thick](y_1)--(x_3);
\draw[thick](z)--(x_6);
\draw[thick](z)--(v);
\draw[thick](y_2)--(x_5);
\draw[thick](y_1) to (x_7);
\draw[thick](x_1) to[out=40, in=140] (x_7);
\draw[thick, dash dot] (-4,1)--(12.5,1);
\draw[thick, dash dot] (-4,3)--(12.5,3);
\draw[thick, dash dot] (-3, 5)--(-3, -1 );
\node(x_1list) at (0,5){\LARGE$\{1\}$};
\node(x_3list) at (4,5) {\LARGE$\{3\}$};
\node(x_2list) at (2,5) {\LARGE$\{2\}$};
\node(x_4list) at (6,5) {\LARGE$ \{1\}$};
\node(x_5list) at (8,5) {\LARGE $\{3\}$};
\node(x_6list) at (9.8,5) {\LARGE$\{2\}$};
\node(x_7list) at (12,5) {\LARGE $\{3\}$};
\end{tikzpicture}
}
\caption{The situation that is described in Case 2 in the proof of Theorem~\ref{t-c4c9}. The set $\{x_6, x_5, x_4, x_3, y,v,z\}$ induces a $C_7$ to which Rule-C7 should have been applied.}\label{f-two}
\end{figure}
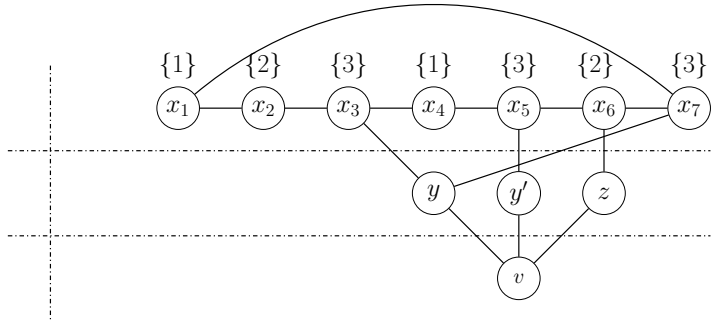

If $yy'\in E$, then $v$ would have a list of size at most~$2$ due to Rule~\ref{r-bull}. Hence $yy'\notin E$.  If $yx_7\notin E$, this means that $\{x_1,x_2,x_3,y,v,y',x_5,x_6,x_7\}$ induces a $C_9$, which is not possible. Hence, $yx_7\in E$.

To summarize, we found that $v$ has two distinct neighbours $y$ and $y'$, where $y$ has exactly two neighbours on $C$, namely $x_3$ and $x_7$, and $y'$ has exactly one neighbour on $C$, namely $x_5$. As $G$ has diameter~$2$, this means that there exists a vertex $z\in N_1$ with $z\notin \{y,y'\}$ that is adjacent to $x_6$ and~$v$. 
Then $z$ is not adjacent to any vertex of $\{x_1,x_3,x_4,x_5,x_7\}$, as otherwise $z$ would have a list of size~$1$ due to Rule~\ref{r-one2}, and by the same rule, $v$ would have a list of size~$2$. 
If $zy\in E$, then $\{y,z,x_6,x_7\}$ induces a $C_4$, which is not possible. Hence $zy\notin E$. 

From the above, we find that $\{x_6,x_5,x_4,x_3,y,v,z\}$ induces a $C_7$; see also Figure~\ref{f-two}. As $c(x_6)=2$, $c(x_5)=3$, $c(x_4)=1$ and $c(x_3)=3$,
we find that $L_c(\{x_6,x_5,x_4\})=\{1,2,3\}$ and $L_c(x_3)=L_c(x_5)$. Then $2\notin L_c(v)$, due to Rule-C7. Hence, $|L_c(v)|\leq 2$, a contradiction.
We conclude that the propagation algorithm returned either {\tt yes} or {\tt no} in Case~2 as well.\qed
\end{proof}

\section{The Proof of Theorem~\ref{t-hard}}\label{a-hard}

\noindent
In this section we prove Theorem~\ref{t-hard}, which we restate below.

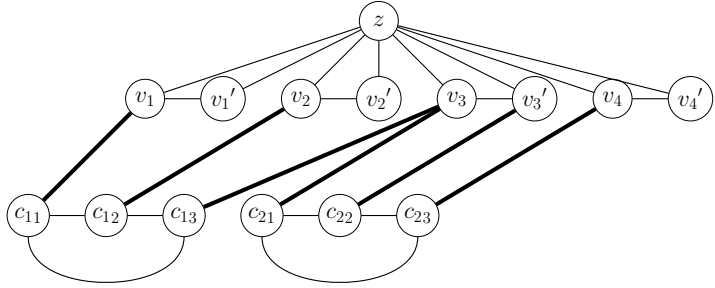
\begin{figure}[h]
\resizebox{9.5cm}{!}{
\begin{tikzpicture}[main_node/.style={circle,draw,minimum size=1cm,inner sep=3pt]}]
\node[main_node](z) at (2,0){\LARGE$z$};
\node[main_node](v_1) at (-4,-2){\LARGE$v_1$};
\node[main_node](v1') at (-2,-2){\LARGE${v_1}^{\prime}$};
\node[main_node](v_2) at (0,-2){\LARGE$v_2$};
\node[main_node](v_2') at (2,-2){\LARGE${v_2}^{\prime}$};
\node[main_node](v_3) at (4,-2) {\LARGE$v_3$};
\node[main_node](v_3') at (6,-2) {\LARGE${v_3}^{\prime}$};
\node[main_node](v4) at (8,-2) {\LARGE$v_4$};
\node[main_node](v4') at (10,-2){\LARGE${v_4}^{\prime}$};
\node[main_node](c11) at (-5,-5) {\LARGE${c_1}_2$};
\node[main_node](c12) at (-7, -5) {\LARGE${c_1}_1$};
\node[main_node](c13) at (-3,-5) {\LARGE${c_1}_3$};
\node[main_node](c21) at (1,-5) {\LARGE${c_2}_2$};
\node[main_node](c22) at (-1,-5) {\LARGE${c_2}_1$};
\node[main_node](c23) at (3,-5) {\LARGE${c_2}_3$};
\draw[thick](z)--(v4);
\draw[thick](z)--(v4');
\draw[thick](z)--(v_1);
\draw[thick](z)--(v1');
\draw[thick](z)--(v_2);
\draw[thick](z)--(v_2');
\draw[thick](z)--(v_3);
\draw[thick](z)--(v_3');
\draw[thick][-,line width=3pt](v_1)--(c12);
\draw[-,line width=3pt](v_2)--(c11);
\draw[-,line width=3pt](v_3)--(c22);
\draw[-,line width=3pt](v_3)--(c13);
\draw[-,line width=3pt](v_3')--(c21);
\draw[-,line width=3pt](v4)--(c23);
\draw[thick](c11)--(c12);
\draw[thick](c11)--(c13);
\draw[thick](c12) to [out=270, in=270](c13);
\draw[thick](c21)--(c22);
\draw[thick](c21)--(c23);
\draw[thick](c22)to [out=270, in=270](c23);
\draw[thick](v_1)--(v1');
\draw[thick](v_2)--(v_2');
\draw[thick](v_3)--(v_3');
\draw[thick](v4)--(v4');
\end{tikzpicture}
}
\caption{An example of a graph $G$ in the reduction from {\sc Not-All-Equal $3$-Satisfiability} to {\sc $3$-Colouring} with clauses $C_1=x_1 \wedge x_2 \wedge x_3$ and $C_2=x_3\wedge \neg x_3 \wedge x_4$. We obtain the graph $G'$ by subdividing the thick edges (edges between literal and clause vertices) the same number of times and connecting the newly introduced vertices to $z$.}\label{f-hardhard}
\end{figure}

\medskip
\noindent
{\bf Theorem~\ref{t-hard} (restated).}
{\it For every even integer $t\geq 6$, {\sc $3$-Colouring} is \NP-complete on the class of $(C_4,C_6,\ldots,C_t)$-free graphs of diameter~$4$.}

\begin{proof}
Note that the problem is readily seen to be in \NP.
To prove \NP-hardness we modify the standard reduction for {\sc Colouring} from the \NP-complete problem {\sc Not-All-Equal $3$-Satisfiability}~\cite{Sc78}, where each variable appears in at most three clauses. 
So, given a CNF formula $\phi$, we first construct a graph~$G$ as follows (see also Figure~\ref{f-hardhard}):

\begin{itemize}
    \item add literal vertices $v_{i}$ and  $v'_{i}$ for each variable $x_i$;
    \item add an edge between each $v_{i}$ and  $v'_{i}$;
    \item add a vertex $z$ adjacent to every $v_i$ and every $v_i'$;
    \item for each clause $C_i$ add a triangle $T_i$ with clause vertices $c_{i_1}, c_{i_2}, c_{i_3}$;
    \item fix an arbitrary order of the literals $x_{i_1}, x_{i_2}, x_{i_3}$ of $C_i$ and for $j\in \{1,2,3\}$, add the edge $v_{i_j}c_{i_j}$ if $x_{i_j}$ is positive and the edge $v_{i_j}'c_{i_j}$ if $x_{i_j}$ is negative.
\end{itemize}

\noindent
It is well known that $\phi$ has a truth assignment $\tau$ such that each clause contains at least one true literal and at least one false literal (call such a $\tau$ satisfying) if and only if $G$ has a $3$-colouring. For completeness we give a proof below.

First suppose $\phi$ has a satisfying truth assignment. Colour vertex $z$ with colour~$1$, each true literal with colour $2$ and each false literal with colour $3$. Then, as each clause has a true literal and a false literal, each triangle $T_i$ has neighbours in two different colours. Hence, we can complete the $3$-colouring.

Now suppose $G$ has a $3$-colouring. Say $z$ is assigned colour $1$. Then each literal vertex has either colour $2$ or colour $3$. Moreover, each $T_i$ must be adjacent to at least one literal vertex coloured~$2$ and to at least one literal vertex  coloured~$3$.  Hence, the truth assignment that sets literals whose vertices are coloured with colour $2$ to be true and those coloured with colour $3$ to be false is satisfying. 

As every clause vertex is adjacent to a literal vertex and literal vertices are adjacent to $z$, every vertex has distance at most $2$ from $z$. So $G$ has diameter~$4$.

We modify $G$ into a graph $G'$: for some $p\geq 0$, subdivide each edge $v_{i_j}c_{i_j}$ and each edge $v_{i_j}'c_{i_j}$ 
$p$ times and make each newly introduced vertex adjacent to $z$; see also Figure~\ref{f-hardhard}. Then $G'$ has a $3$-colouring if and only if $G$ has a $3$-colouring, as the new vertices will be alternatingly coloured by $2$ and $3$ if $z$ has colour~$1$.
Moreover, $G'$ still has diameter~$4$, and it can be readily checked that every induced cycle of~$G$ of length at most $p$ is either a $C_3$ (either a triangle $T_i$ or a triangle containing $z$) or a $C_5$ (which must contain~$z$). As we can make $p$ arbitrarily large, the result follows.\qed
\end{proof}

\section{Conclusions}\label{s-con}

We proved that {\sc $3$-Colourability} is polynomial-time solvable for several subclasses of diameter~$2$ that are characterized by forbidding one or two small induced cycles. In order to do this we used a unified framework of propagation rules, which allowed us to exploit the diameter-$2$ property of the input graph. 
Our current techniques need to be extended to obtain further results (in particular, we cannot currently handle the increasing number of different $3$-colourings of induced cycles of length larger than~$9$).

As open problems we pose:
determine the complexity of $3$-{\sc Colouring} and {\sc List $3$-Colouring} for:
\begin{itemize}
\item graphs of diameter~$2$ (which we recall is a long-standing open problem)
\item $C_t$-free graphs of diameter~$2$ for $s\in \{3,4,7,8,\ldots\}$; and 
\item 
$(C_4,C_t)$-free graphs of diameter~$2$ for $t\geq 10$.
\end{itemize} 
We also note that the complexity of $k$-{\sc Colouring} for $k\geq 4$ and {\sc Colouring} is still open for $C_3$-free graphs of diameter~$2$ (see also~\cite{MPS19}).

Finally, we turn to the class of graphs of diameter~$3$. The construction of Mertzios and Spirakis~\cite{MS16} for proving that {\sc $3$-Colouring} is \NP-complete for $C_3$-free graphs of diameter~$3$  appears to contain not only induced subdivided stars of arbitrary diameter and with an arbitrary number of leaves but also induced cycles of arbitrarily length $s\geq 4$. Hence, we pose as open problems:
determine the complexity of $3$-{\sc Colouring} and {\sc List $3$-Colouring} for $C_t$-free graphs of diameter~$3$ for $t\geq 4$
and $(C_4,C_t)$-free graphs of diameter~$3$ for $t\in \{3,5,6,\ldots\}$.

\end{document}